\documentclass[12pt]{amsart}

\usepackage{amssymb, amsmath, fancyhdr, amsthm, bm, graphicx}
\usepackage[ps,matrix,curve,frame,arrow,rotate,line]{xy}

\newcommand{\ZZ}{\mathbb{Z}}

\newcommand{\RR}{\mathbb{R}}
\newcommand{\QQ}{\mathbb{Q}}

\newcommand{\floor}[1]{\left\lfloor #1 \right\rfloor}

\newcommand{\lra}{\longrightarrow}

\newcommand{\Bl}[2]{B_{#1}\left(#2\right)}

\renewcommand{\phi}{\varphi}

\newcommand{\cbc}{Mahler-Bernoulli class}

\newcommand{\shift}{\mathcal{S}}

\newcommand{\qp}{\ensuremath{\text{``}p\text{''}}}
\theoremstyle{plain}
\newtheorem{theorem}{Theorem}
\newtheorem{lemma}[theorem]{Lemma}
\newtheorem{prop}[theorem]{Proposition}
\newtheorem{rem}[theorem]{Remark}
\newtheorem{cor}[theorem]{Corollary}
\theoremstyle{definition}
\newtheorem{defn}[theorem]{Definition}
\newtheorem{ex}[theorem]{Example}

\numberwithin{theorem}{section}
\title[$p$-adic Shift and Applications]{Dynamics of the $p$-adic Shift and Applications}

\author{James Kingsbery}
\address{Williams College, Williamstown, MA 01267}
\email{06jck@williams.edu}
\author{Alex Levin}
\address{Massachusetts Institute of Technology, Cambridge, MA, 02138}
\email{levin@mit.edu}
\author{Anatoly Preygel}
\address{Massachusetts Institute of Technology, Cambridge, MA, 02138}
\email{preygel@mit.edu}
\author{Cesar E. Silva}
\address{Williams College, Williamstown, MA 01267}
\email{csilva@williams.edu}

\begin{document}
\maketitle

\section{Introduction}\label{sec:intro}
In recent years, several authors have started studying the dynamics that 
result from various maps  on the $p$-adics.  In many cases they have shown that
relatively simple and natural  transformations satisfy important dynamical properties, 
such as ergodicity. 

In this paper, we continue this line of research by  studying
  Bernoulli transformations  on the $p$-adic integers.   Bernoulli 
transformations are those that can be identified with the ``left shift'' of infinite (to the right) sequences on some alphabet. They are ubiquitous throughout the field of dynamics, and come
up in numerous guises in several branches of mathematics.    In the realm of measurable dynamics, where we are free to disregard sets of measure zero, 
the map $T_n:[0,1] \to [0,1]$ for $n$ a positive integer, 
taking   $x$ to $n x\bmod1$ (i.e., the fractional part of $nx$), provides a simple example of a Bernoulli (noninvertible) transformation.  Taking base-$n$ expansions of $x$, and ignoring the (measure zero) set of redundant expansions $.\overline{n-1}$, one sees that indeed this map is just a left shift.

Moving to the $p$-adic context, the aim of this paper is to   present a novel way of realizing
the Bernoulli shift on $p$ symbols,  where $p$ is some prime.
 We do this
by starting with the most natural realization: Any $x\in \ZZ_p$ has a unique (possibly infinite) 
expansion of the 
form $x=\sum_{i=0}^{\infty} b_i p^i$ ($b_i\in\{0,1,\ldots,p-1\}$); one can define the ``$p$-adic shift'' 
$\shift:\ZZ_p \to \ZZ_p$ to be a left shift on this expansion, that is $\shift(x) = \sum_{i=0}^{\infty} 
b_{i+1} p^i$.  By showing that suitably small perturbations of $\shift$ are still Bernoulli, we can find 
many ``nice'' maps, such as polynomials, that behave like the shift map $\shift$ in this way. 
(We originally discussed these maps in \cite{polybern}; this paper presents a novel and more direct way
of obtaining them, as we will  remark below.) 
Because we are
working on the $p$-adics,  our proofs use divisibility
properties of  certain polynomials, and thus, 
we obtain a connection between dynamics and number theory.  

In addition to allowing us to obtain aesthetically pleasing and natural
representations of Bernoulli shifts and connecting the study to number-theoretic
properties like divisibility, working in the $p$-adic setting 
is appealing; dynamics over the $p$-adics
 has recently become an active area of research.  Our work was motivated by the article
 \cite{bs05}, which studied the measurable dynamics of polynomial 
maps on $\ZZ_p.$  The authors in  \cite{bs05} asked when polynomial 
maps can satisfy the measurable dynamical property of being mixing.   
 (It turns out that this was known to Woodcock and Smart,
who   showed in \cite{WoodcockSmart}  that the polynomial 
 map $x \mapsto \tfrac{x^p-x}{p}$ defines a Bernoulli, hence mixing,  
 transformation on $\ZZ_p$.)  In \cite{polybern}, the authors gave a 
 detailed account of the  dynamics that can result from a certain 
 well-behaved class of maps on the $p$-adics. In particular, we introduced a set of conditions on the Mahler expansion of a transformation on the $p$-adics which are sufficient for it to be Bernoulli (see Definition~\ref{cbc}).

 The sufficiency of these conditions was proved in \cite{polybern} via a structure theorem for so-called ``locally scaling'' transformations (see Definition~\ref{defn:ls}).   The {\cbc} conditions mentioned above were (roughly) that the transformation be a small enough perturbation of $x \mapsto \binom{x}{p}$ (which are part of the Mahler basis), and so applying the structure theorem required a careful study of the dynamics of the map $x \mapsto \binom{x}{p}$.  The approach of the present paper turns out to be more direct than that of \cite{polybern}, to which it is in a sense dual: we begin with the easily understood dynamics of $\shift$ and work to better understand its Mahler expansion, proving that it is ``a small enough perturbation'' of $x \mapsto \binom{x}{p}$.  Since we are working in an ultrametric setting, this means that the {\cbc} condition may be reformulated as exactly those things that are small enough perturbations of $\shift$, giving a more conceptual reason why they should be Bernoulli!

This paper is organized as follows. After a brief review of the 
$p$-adics, we begin by studying the $p$-adic shift and its 
properties. 
 Section~\ref{sec:perturb} is devoted to developing our 
machinery and proving the main result. We introduce locally scaling transformations (similarly to 
our discussion in \cite{polybern}) and then show that 
what we mean when we say that certain maps are small perturbations of the shift map. 
 At the end of the section,
we prove that certain locally scaling transformations are indeed small perturbations of the 
shift map, and hence Bernoulli.  In particular, we define the {\cbc} and show that 
all members thereof satisfy these properties.   Finally,  in 
Section~\ref{sec:related} we briefly talk about maps related to the 
$p$-adic shift.

We have discussed locally scaling transformations and their dynamical properties 
in \cite{polybern}.  However, the interpretation that certain locally scaling transformations 
are in some sense small perturbation on the shift map, and therefore remain Bernoulli,
is original to this paper (and is its main contribution).

Some basic familiarity 
with the $p$-adics, as provided, for example, in \cite{gouvea} or 
\cite{robert}, will be helpful. The monograph \cite{Khr04} may be consulted for an 
introduction to $p$-adic dynamics. For measurable dynamics properties such as mixing the 
reader may consult \cite{Silva08}. Finally, 
\cite{silverman} introduces more extensive connections between dynamics and arithmetic.

 \subsection*{Acknowledgments}
This paper is based on research by
the Ergodic Theory group of the 2005 SMALL  summer
research project at Williams College.
Support for the project was provided by National Science Foundation
REU Grant DMS -- 0353634 and the Bronfman Science Center of Williams College.

\section{The $p$-adics}\label{padicintro}

Throughout the paper, we fix a prime $p$.  
\subsection{The $p$-adic integers $\ZZ_p$}
The $p$-adic integers generalize the notion of base-$p$ expansion
of nonnegative integers.
For any nonnegative integer $n,$ we can write down a base-$p$  
expansion of $n,$ i.e., an expression $n=b_0 + b_1 p + b_2 p^2 
+\cdots+b_N p^N$ where the $b_i$ are in $\{0,1,\ldots,p-1\}.$
Thus, the nonnegative
integers are precisely the finite power series in $p$ with coefficients in $\{0,1,\ldots,p-1\}.$  

The $p$-adic integers $\ZZ_p$ can be thought of as infinite series in 
a formal variable $\qp,$ with the same coefficients as above: 
i.e., as  a set we can define
\[ \ZZ_p \stackrel{\text{def}}{=} \left\{ a_0 + a_1 \qp + \cdots + 
a_i {\qp}^i + \cdots = \sum_i a_i {\qp}^i : a_i \in 
\{0,1,\ldots,p-1\} \text{ for all $i \geq 0$}\right\}. \]  We can 
turn $\ZZ_p$ into a metric space by introducing  a distance function $d: 
\ZZ_p \times \ZZ_p \to \RR,$ where
\[ d\left( \sum_i a_i {\qp}^i, \sum_i b_i {\qp}^i\right) = 
\begin{cases} 0 & \text{if $a_i = b_i$ for all $i \geq 0$} \\ p^{-k} 
& \text{ where $k=\min\{i \geq 0: a_i \neq b_i\}$, otherwise} 
\end{cases} \]
In other words, the distance between two points is small if their 
their expansions agree to many terms. 
One checks that $d(\cdot,\cdot)$ is indeed a metric on $\ZZ_p$, 
making it a complete metric space.  We may also define a function 
$\left| \cdot \right|: \ZZ_p \to \RR$ by
\[ \left| \sum_i a_i {\qp}^i \right| = d\left(\sum_i a_i {\qp}^i, 
\sum_i 0 {\qp}^i\right) = \begin{cases} 0 & \text{if $a_i = 0$ for 
all $i \geq 0$} \\ p^{-k} & \text{where $k=\min\{i \geq 0: a_i \neq 
0\}$, otherwise} \end{cases} \]  This will be our discrete valuation, 
though we can't make sense of this notion without defining the ring 
structure.  We will alternately refer to $\left|\cdot\right|$ as the 
$p$-adic absolute value. It is not hard to see that for any $u,v\in 
\ZZ_p,$ we have
$d(u,v)=|u-v|.$

\begin{rem}We note that  $\ZZ_p$ is also often constructed as 
the metric space completion of $\ZZ$ with respect to the $p$-adic 
valuation $\left| \cdot \right|$.  However, we would still need the 
``base $p$ expansions,'' so might as well start from this viewpoint.
\end{rem}

We now define the ring structure on $\ZZ_p.$ 
To do this, we  consider the functions $\pi_k: \ZZ_p \to \ZZ/p^k \ZZ$ 
given by
\[ \pi_k\left( \sum_i a_i {\qp}^i \right) = \left(\sum_{i=0}^{k-1} 
a_i p^i \right) + p^k \ZZ. \]  Observe that they are compatible with 
the natural projections $\ZZ/p^k \ZZ \to \ZZ/p^\ell \ZZ$ for $\ell < 
k$, and that
\[ \bigcap_{k \geq 0} \pi_k^{-1}\left(\pi_k\left( \sum_i a_i {\qp}^i 
\right)\right) =\left\{ \sum_i a_i {\qp}^i \right\}. \]
This allows us to conclude that there is a unique ring structure on 
$\ZZ_p$ for which each $\pi_k$ is a ring homomorphism.  (In fact, 
 readers familiar with  inverse limits can see
that   the 
$\pi_k$ realize $\ZZ_p$ as the inverse limit 
$\mathop{\varprojlim}_{k} \ZZ/p^k \ZZ$ where $\ZZ/p^k \to 
\ZZ/p^{\ell}$ is the natural projection for $\ell < k$.)

The verification of this fact is not hard from the above two 
observations, so we will simply describe the ring operations rather 
than give a formal proof: To compute the first $k$ coefficients of 
\[ \left(\sum_i a_i {\qp}^i  \right) + \left( \sum_i b_i {\qp}^i 
\right) \qquad \left(\text{resp., }  \left(\sum_i a_i {\qp}^i  
\right) \cdot \left( \sum_i b_i {\qp}^i \right) \right) \] one 
applies $\pi_k$ to each term, adds (resp., multiplies) the resulting 
elements of $\ZZ/p^k \ZZ$, and takes the first $k$ digits of the base 
$p$-expansion of any representative class of the result (noting that 
the first $k$ digits are independent of choice of representative). 

With this description, we see that the function $\iota: \ZZ \to 
\ZZ_p$ given by taking an integer to its formal base $p$ expansion is 
precisely the unit ring homomorphism $n \mapsto n \cdot 1$.  So, 
suppose $\sum_i a_i {\qp}^i$ is such that $a_i = 0$ for $i \gg 0$.  
Then, the sum $\sum_i a_i p^i$ makes sense as an integer and \[ 
\sum_i a_i {\qp}^i = \iota\left( \sum_i a_i p^i \right) = \sum_i 
\iota(a_i) \iota(p)^i. \]

Suppose now that $\sum_i a_i {\qp}^i \in \ZZ_p$ is arbitrary.  It is 
the limit, in the metric space topology, of its truncated expansions 
(formal partial sums).  Thus we can, and will, forever remove the 
quotes around ${\qp}$ and regard the infinite sum as a convergent sum 
rather than a formal expression, all with no ambiguity.  Thus, we 
have defined $\ZZ_p$ as a complete discrete valuation ring; it
is not hard to see that $\ZZ_p$ has $\ZZ$ as a dense subset. 

We 
may also wish to ask what the units in $\ZZ_p$ are (i.e., for what 
$u\in\ZZ_p$ does there exist a $v$ in $\ZZ_p$ such that $uv=1$). It 
is not too difficult to show that $\ZZ_p^\times=\{u\in 
\ZZ_p:|u|=1\},$ i.e., all $p$ adic integers with nonzero $p^0$ 
coefficient.

\subsection{Some $p$-adic properties}
As a metric space,  $\ZZ_p$  satisfies some unusual properties. First 
of all, we have the \emph{strong triangle inequality}: if $u,v,w\in 
\ZZ_p,$ then it is the case that $|u-v|\leq\max(|u|,|v|).$ From this, 
it is not hard to see that if $|u|>|v|,$ then $|u+v|=|u|.$
Furthermore, this implies that a series $\sum a_m$ converges if and 
only if $|a_n|\to 0.$

In our work, balls in the $p$-adic integers play a fundamental role. 
We define the (closed) ball of radius $r$ about $u\in \ZZ_p$ as  
$\Bl{r}{u}=\{v\in\ZZ_p:|u-v|\leq r\}.$ Because $p$-adic distances are 
all powers of $p,$ we may take $r=p^i$ for some $i$ without loss of 
generality.  Notice that if  $u=a_0+a_1 p + a_2 p^2 + 
\cdots\in\ZZ_p,$ then $\Bl{p^{-i}}{u} = a_0 + a_1 p +a_2 p^2 
+\cdots+a_{i-1}p^{i-1} + p^i \ZZ_p,$ i.e., $\Bl{p^{-i}}{u}$ consists of
 all elements of $\ZZ_p$ 
agreeing with the $p$-adic expansion of $u$ to the $p^{i-1}$  term. 
Using the strong triangle inequality, it is easy to see that if   $u' 
\in \Bl{p^i}{u},$ then $\Bl{p^{-i}}{u}=\Bl{p^i}{u'}$ (thus ``every point 
in a ball is its center''), and further, that if two balls intersect, 
then one is contained in the other. The $p$-adic integers form the 
unit ball 
around any element of $\ZZ_p.$

Additionally, it is not difficult to see that for $j \geq i,$ each 
ball of radius $p^{-j}$ contains $p^{j-i}$ disjoint balls of radius 
$i.$
In particular, each ball is totally bounded,
and being closed, is compact.

\subsection{Brief introduction to $\QQ_p$}
The set $\QQ_p$ is the field  of fractions
of $\ZZ_p,$ i.e.,   $$\QQ_p \stackrel{\text{def}}{=}\ZZ_p[1/p]=\bigcup_{i\geq0}p^{-i}\ZZ_p.$$
Any element of $\QQ_p$ can be written as a Laurent series in $p,$ 
with finitely many negative powers but possibly infinitely many 
positive ones.  
Alternately, just like $\ZZ_p$ is the completion of 
$\ZZ$ under the $p$-adic metric, 
$\QQ_p$ is the completion of $\QQ$ under this same metric. 
(To define $|\cdot|$ on $\QQ,$ note that any element of $\QQ$ can be 
written uniquely as $s=p^\ell a/b,$ where $p$ divides neither
$a$ nor $b$; in this case, $|s| = p^{-\ell}$.) 

The set $\QQ_p$ possesses many of the same properties
as $\ZZ_p$ (for example, the strong triangle inequality).

\subsection{Examples}
\begin{ex}
In the $2$-adics, $1+2+4+8+\cdots=1/(1-2)=-1.$ Indeed, we can see 
that the difference between each successive partial sum and $-1$ 
becomes divisible by increasing powers of $2,$ and consequently, 
becomes ``smaller.''

What if we had wanted to derive the $2$-adic representation of $-1$? 
Notice that for each $m,$ we have 
$-1=1+2+2^2+\cdots+2^{m-1}\bmod2^m.$ The complete $2$-adic expansion 
follows.
\end{ex}
\begin{ex}
Let us try to find a zero of the polynomial $q(x) = x^2 + 1$ in $\ZZ_5.$ 
Notice that $q(2) = 5 \in \Bl{1/5} {0}.$ 
So, $x=2$ is an approximation  to a zero of $q.$

How do we get a better approximation?  More precisely, we would like to find a ``small'' 
$t$ such that $q(2+t) \in \Bl{1/25}{0}.$ By ``small,'' we mean small in the $5$-adic sense; in
this case, we will require $t\equiv0\bmod5.$ 

Let us write down  the formal Taylor  
series $q(2+t) = q(2) + t q'(2) + \cdots,$ where the dots represent terms 
of order at least $2$ in $t$ (hence, a contribution that is equal to $0$ mod $25$).  We  notice
 that $q'(2) \neq0\bmod 5,$ and so, we see that $$t=\frac{q(2+t)-q(t) }{q'(2)}+\cdots.$$ We
 require $q(2+t)\equiv0\bmod25,$ and know that $q(2) \equiv 5\bmod 25.$  Because $q'(2)=4\not\equiv 0\bmod5,$
 this determines $t$ modulo $25$; the concealed terms  are all $0$ modulo $25,$
and thus do not affect the result
at this stage.   In particular, $t\equiv 20 /4\bmod25,$ so $t\equiv 5\bmod25.$ 
And indeed, $q(2+1\cdot5) \equiv 0 \bmod25.$ 
 
 Proceeding in this way, we can find $a_0,a_1,a_2,\ldots\in\{0,1,2,3,4\}$ 
 such that 
 $q(a_0+a_1 5 +a_2 5^2 +\cdots +a_k 5^k)\equiv0\bmod 5^{k+1}.$ It follows
 that for $\alpha = a_0 + a_1 5+a_2 5^2 +\cdots\in \ZZ_5,$ we have $q(\alpha) = 0,$
since the partial sums provide successively better approximations for the zero of $q.$  

This example demonstrated a  specific instance of \emph{Hensel's Lemma}, a general result
about finding zeros of polynomials in $\ZZ_p$ \cite{gouvea}.

\end{ex}

\section{The $p$-adic Shift}
The $p$-adic shift  $f:\ZZ_p\lra\ZZ_p$ is defined as follows. If
$x=b_0+b_1p+b_2p^2+\cdots,$ where the $b_i\in\{0,1,\ldots,p-1\},$ we 
let $\shift(x)=b_1+b_2p+b_3p^2+\cdots.$ We immediately see that  if $\shift^k$
denotes the $k$-fold iterate of $\shift,$ then we have that 
$\shift^k(x)=b_k+b_{k+1}p+\cdots.$ Moreover, 
 for $x\in\ZZ,$ it is the case that $\shift^k(x)=\floor{x/p^k}$ 
where $\floor{\cdot}$ is the greatest integer 
function. 

One of the most basic properties of the shift is that it is
continuous as a function of $\ZZ_p.$ Indeed, if $|x-y|<1/p^{k+1},$ it 
is not hard to see that 
$|\shift(x)-\shift(y)|<1/p^k.$ (Recall that $\left|\cdot\right|$ is the $p$-adic
absolute value.)

Continuity is important because by Mahler's Theorem, 
any continuous $T:\ZZ_p\lra\ZZ_p$ can be expressed in the form of a 
uniformly convergent series, called its \emph{Mahler 
Expansion}:
\begin{equation}
T(x)=\sum_{n=0}^\infty a_n \binom{x}{n}
\end{equation} 
where 
\begin{equation}\label{eq:findiff}a_n=\sum_{i=0}^n(-1)^{n+i}T(i){n\choose
    i} \in \ZZ_p\end{equation}
\begin{rem}
For an arbitrary function $T:\ZZ_p\to\ZZ_p,$ we can write down the 
formal identity 
$T(x)=\sum_{n=0}^\infty a_n{x\choose n}$ for coefficients 
$a_n\in\ZZ_p$ to be determined.  Then, substituting in 
$x=0,1,2,\ldots$ in turn, we can inductively determine that the $a_n$ 
would have to be as  in 
(\ref{eq:findiff}), noting that only finitely many summands will be 
non-zero at each stage.  The true content of  Mahler's Theorem is 
that for $T$ continuous on $\ZZ_p,$ this series converges, which 
happens if and only if   $|a_n|\to0$ as $n\to\infty$ by convergence properties on 
the $p$-adics.   
\end{rem}

Our first goal is to study the Mahler expansion of $\shift^k.$ Throughout this paper,
we let $a_n^{(k)}$ be the $n^{\text{th}}$ Mahler 
coefficient 
of $\shift^k.$ In other words, the $a_n^{(k)}$ are defined such that 
\begin{equation}\shift^k(x) = \sum _{n=0}^\infty a_n^{(k)} \binom{x}{n}.\end{equation} 

\begin{theorem}\label{thm:mahlercoeffsTOLY}
The coefficients $a_n^{(k)}$ satisfy the following properties:
\begin{enumerate}
\item[(i)] $a_n^{(k)} = 0$ for $0 \leq n < p^k$;
\item[(ii)] $a_n^{(k)} = 1$ for $n = p^k$;
\item[(iii)] Suppose $j \geq 0$.  Then, $p^j$ divides $a_n^{(k)}$ for 
$n>j p^k - j + 1$ (and so, $|a_n^{(k)}| \leq 1/p^j$).
\end{enumerate}
\end{theorem}
\begin{proof}
Our proof is closely based on Elkies's short derivation of Mahler's 
Theorem \cite{elkiesmahler}. 
Let
\[ F(t) = \sum_{n \geq 0} \shift^k(n) t^n \in 
\ZZ_p[[t]]\qquad\text{and}\qquad A(u) = \sum_{n \geq 0} a_n^{(k)} u^n 
\in \ZZ_p[[u]]. \]
Recall the standard power series identities \[ \frac{1}{1-t} = 
\sum_{i \geq 0} t^i \qquad\text{and}\qquad\frac{t}{(1-t)^2} = \sum_{i 
\geq 1} i t^i.\]  Using them and the definition of $\shift^k,$ we may 
compute that
\begin{align*}
F(t) &= \sum_{n \geq 0} \shift^k(n) t^n =  \sum_{a=0}^{p^k-1} \sum_{b \geq 
0} b t^{a+b p^k} \\
&= \left(\sum_{a=0}^{p^{k}-1} t^a \right)\left(\sum_{b \geq 0} b 
(t^{p^k})^b \right) = \left(\frac{1-t^{p^k}}{1-t}\right) 
\left(\frac{t^{p^k}}{(1-t^{p^k})^2}\right) \\
&= \frac{t^{p^k}}{(1-t)(1-t^{p^k})}
\end{align*} 

Before proceeding, we remark that   \begin{equation}A(u) = 
\frac{1}{1+u} 
F\left(\frac{u}{1+u}\right).\label{eq:A(u)}\end{equation}  Indeed,  
 suppose $T: \ZZ_p \to \QQ_p$ is any map.  Set $\tilde F(t) = \sum_{n 
 \geq 0} T(n) t^n$ and  $\tilde A(u) = \sum_{n \geq 0} a_n u^n$ where 
 $a_n$ is such that $\sum_i a_i \binom{n}{i} = T(n)$ for all $k \geq 
 0$.  Then, 
\begin{align*}\tilde F(t) &= \sum_{n \geq 0} \sum_{i=0}^{n} a_i 
\binom{n}{i} t^n = \sum_{i \geq 0} a_i \sum_{n \geq i} \binom{n}{i} 
t^n 
= \sum_{i \geq 0} a_i \frac{t^i}{(1-t)^{i+1}} \\
&= \left(\frac{1}{1-t}\right) \sum_{i \geq 0} a_i 
\left(\frac{t}{1-t}\right)^i = \frac{1}{1-t} \tilde 
A\left(\frac{t}{1-t} \right) \end{align*}
From this, (\ref{eq:A(u)}) follows by taking $\tilde F=F,$ $\tilde 
A=A,$ and $t=u/(1+u).$

Now, note that because $p|\binom{p^k}{i}$ for all $0<i<p^k,$ we have  
$(1+u)^{p^k} - u^{p^k} = 1 + p R(u)$ where $R(u)$ is a polynomial of 
degree $p^k-1$ in $u$ without leading term, so that
\[ A(u) =  \frac{1}{1+u} F\left(\frac{u}{1+u}\right) = 
\frac{u^{p^k}}{(1+u)^{p^k}-u^{p^k}} = u^{p^k}\left(1 + p R(u) + p^2 
R(u)^2 + p^3 R(u)^3 + \cdots \right) \]  Since $R(u)$ has no leading 
term, we can now conclude (i) and (ii).  The case $j=0$ of (iii) is 
trivial, so we may assume $j \geq 1$.  Working modulo $p^j$ we obtain 
the equality (in $\ZZ_p/p^j \ZZ_p[[u]]$)
\[ A(u) \equiv u^{p^k} \left(1 + p R(u) + \cdots + p^{j-1} 
R(u)^{j-1}\right) \pmod{p^j} \]  Note that the right hand side is a 
polynomial of degree $p^k + (j-1)(p^k-1) = j p^k -j + 1$.  This 
allows us to conclude (iii).
\end{proof}

As we will see, the next corollary will be of fundamental importance 
in the following section, 
where we try to bound coefficients in the Mahler expansions of maps 
related to the shift. 
\begin{cor}\label{cor:shiftcoeffs}
The maximum possible value for $p^{\floor{\log_p n}}|a_n^{(k)}|$
is $p^k$ and it is attained only when $n=p^k.$
\end{cor}
\begin{proof}
We may assume $a_n^{(k)} \neq 0$.  Let $v_p(a_n^{(k)})$ denote the 
integer $\ell$ so that $p^\ell$ precisely divides $a_n^{(k)}$.  We 
are asked to prove that $\floor{\log_p n} - v_p(a_n^{(k)}) \leq k$ 
with equality if and only if $n = p^k$.  Setting $\ell = 
\floor{\log_p n} - k$, we are thus to show that $p^\ell$ divides 
$a_n^{(k)}$ and $p^{\ell+1}$ divides it unless $n=p^k$.

Note that $p^\ell \geq \ell$ so that \[ n \geq p^{\floor{\log_p n}}= 
p^\ell p^k \geq \ell p^k \geq \ell p^k - \ell + 1 \] and applying 
Theorem~\ref{thm:mahlercoeffsTOLY} yields that $p^{\ell}$ does in 
fact divide $a_n^{(k)}$.

If $\ell > 0$, then $p^\ell \geq \ell+1$ so that we in fact have $n 
\geq (\ell+1) p^k \geq (\ell+1) p^k - (\ell+1) + 1$ and so 
Theorem~\ref{thm:mahlercoeffsTOLY} yields that $p^{\ell+1}$ does in 
fact divide $a_n^{(k)}$.  Since $a_n^{(k)}=0$ for $n < p^k$, by 
Theorem~\ref{thm:mahlercoeffsTOLY}, it suffices to prove that $p$ 
divides $a_n^{(k)}$ for $n > p^k$; but this is immediate from 
Theorem~\ref{thm:mahlercoeffsTOLY}.
\end{proof}

\section{Perturbing the Shift}\label{sec:perturb}
\subsection{Local Scaling} The  shift $\shift$ cuts off the first digit 
term in the $p$-adic expansion
of $x\in\ZZ_p.$ Notice that if the expansions of $x$ and $y$ agree
past the     the first digit (i.e., $|x-y|\leq1/p$), $\shift$ multiplies 
the distance between them
by $p.$ Thus, $\shift$ scales distances between points that are close
enough. 

This observation motivates the following definition from \cite{polybern}.
\begin{defn}\label{defn:ls}
We say that $T:\ZZ_p\lra\ZZ_p$ is \emph{locally scaling} with scaling
radius $r$ and scaling constant $C\geq1,$ and write that $T$ is
$(r,C)$ locally scaling   if for all $x,$ $y$ with $|x-y|\leq r,$ 
$|T(x)-T(y)|=C|x-y|.$ We will always  assume, without loss of 
generality, that $r=p^\ell$ and $C=p^m,$ where  $\ell\leq0$   
and $m\geq0.$
\end{defn} 

\begin{prop}
The map $\shift^k$ is $(p^{-k},p^k)$ locally scaling.
\end{prop}
\begin{proof}
Immediate.
\end{proof}
Notice that if $T $ is a  locally scaling map, it is  continuous, and 
for $r'\leq r,$
the restriction $T\left.\right|_{\Bl{r'}{x}}$ is injective into
$\Bl{Cr'}{T(x)}.$  It is  
surjective as well.

\begin{prop}
For $T$ an $(r,C)$ locally scaling map and $r'\leq r,$ the restricted 
map
$T\left.\right|_{\Bl{r'}{x}}:\Bl{r'}{x}\lra\Bl{Cr'}{T(x)}$ is a 
bijection.
\end{prop} 
\begin{proof}
Let $S=T\left.\right|_{\Bl{r'}{x}}.$ Because injectivity of $S$ is 
clear, we just prove surjectivity. We first
demonstrate that  the image of $S$ is dense in $\Bl{Cr'}{T(x)}$ by
showing that for every ball $B\subset\Bl{Cr'}{T(x)}$ there exists
$w\in\Bl{r'}{x}$ such that $T(w)\in B.$ Say the radius of $B$ is
$p^{-j}\leq Cr'$ for $j\in\ZZ_{\geq0}.$ 

Assume furthermore that  there are $\eta$ balls of radius 
$(1/C)\,p^{-j}$ contained in $\Bl{r'}{x}.$
 Pick one point in each of these balls of radius $(1/C)\,p^{-j}$      
 (the \emph{representative} of that ball). Because these $\eta$ 
 representatives are all at least $(1/C)\,p^{-j}$ apart from one 
 another and $(1/C)\,p^{-j}\leq r',$ their images have to be 
at least $p^{-j}$ apart from one another by local scaling. Thus, they
 have to occupy $\eta$ distinct  balls of radius $p^{-j}$ in the
 range. 
Finally, because the number of balls of radius $p^{-j}$ contained in $
\Bl{Cr'}{T(x)}$ is also $\eta,$  one of the representatives in 
$\Bl{r'}{x}$ must map into $B$ by the pigeonhole principle. 

Now that we have the image of $S$ being dense in $\Bl{Cr'}{T(x)},$ we 
note that $S$ is continuous and $\Bl{r'}{x}$ is compact, so 
$S(\Bl{r'}{x})$ must be compact hence closed. 
Thus it is all of $\Bl{Cr'}{T(x)}.$ Hence, our map is surjective and 
therefore bijective.
\end{proof}

\subsection{$(p^{-k},p^k)$ locally scaling maps}
We now turn our attention to a special case of local scaling, 
the $(p^{-k},p^k)$ locally scaling maps. 
We will see that they behave ``nicely'' under preimages.
More precisely, when one takes successive preimages of a given ball, 
we get a union of smaller balls that are very evenly distributed 
throughout $\ZZ_p.$  

Let $T$ be such a map. In this case, each of the $p^k$ balls of 
radius $p^{-k}$ is 
mapped bijectively onto $\ZZ_p.$ It follows that given a ball 
$B\subset\ZZ_p,$ of radius $p^{-j}$ its preimage is 
$$T^{-1}(B)=\bigsqcup_{i=1}^{p^k} B_i$$ where each $B_i$ is a ball of 
radius $p^{-(j+k)},$ and the $B_i$ are contained in distinct balls of 
radius $p^{-k}.$  Furthermore, if $B'\subset B$ is a ball of radius 
$p^{-j'}\leq p^{-j},$ then 
$$T^{-1}(B')=\bigsqcup_{i=1}^{p^k} B_i'$$ where each $B_i'$ is a ball 
of radius $p^{-(j'+k)},$ and each $B_i'$ is contained in $B_i.$ We 
call this very nice property of preimages the \emph{nesting property} 
(see Figure \ref{fig:nesting} for an illustration). 

\begin{figure}
\includegraphics[width=3in]{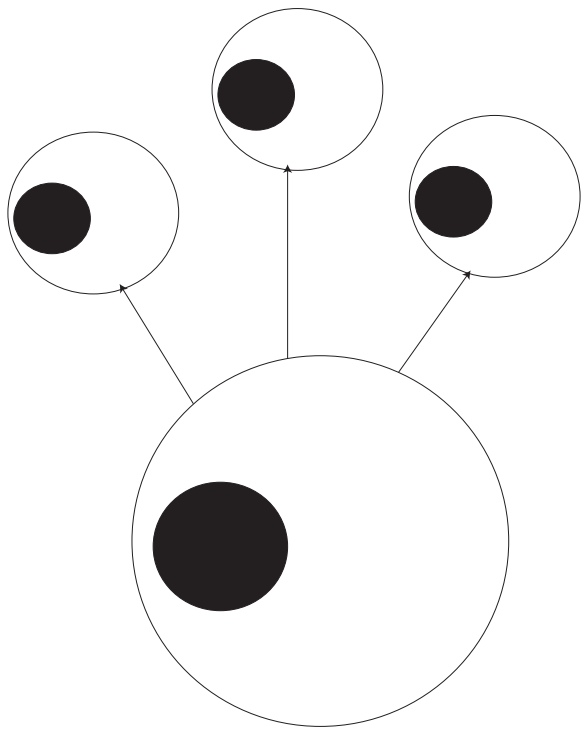}
\caption{\label{fig:nesting} An illustration of the local nesting 
property.  Let 
$B$ be the big ball at the bottom of the figures; the arrows lead to 
its preimages $B_1,$ $B_2,$ and $B_3.$ The shaded ball inside $B$ is 
some ball $B'.$ Its preimages $B_1',B_2'$ and $B_3'$ each fall within 
one of the $B_i.$}
\end{figure}

Finally, the next two lemma will play a role in our main result.
\begin{lemma}\label{lem:bcforballsinsideballs}
Suppose that $B$ is a ball of radius $p^{-k j}.$ Then, $T^{-1}(B) $ 
consists of a union
of $p^{k}$ balls of radius $p^{-k(j+1)},$ each of them in a distinct 
ball of radius $p^{-kj}.$ 
\end{lemma}
\begin{proof}
The preceding discussion proves everything except the last statement. 
So, say we have $x$ and $y$ in $T^{-1}(B),$ so that both are in the 
same ball of radius $p^{-kj}.$ Thus, $|x-y|\leq p^{-kj},$  so 
$|T(x)-T(y)|=p^{k}|x-y|.$ But, it must be the case that 
$|T(x)-T(y)|\leq p^{-kj},$ and this only happens when $|x-y|\leq 
p^{-k(j+1)}.$  It follows that $x $ and $y$ are both in the same ball 
of $T^{-1}(B).$
\end{proof}

\begin{lemma}\label{lem:ballsinsideballs}
Let $T:\ZZ_p\to\ZZ_p$ be $(p^{-k},p^k)$-locally scaling. Then, for 
any $j,n\geq1,$ we have that $$T ^{-n}(\Bl{p^{-jk}}{y}) = \bigcup 
_{\nu=1}^{p^{kn}} B _ \nu,$$ where the $B_\nu$ are balls of radius 
$p^{-k(n+j)},$ one inside each ball of radius $p^{-kn}$ in $\ZZ_p.$
\end{lemma}
\begin{proof}
We proceed by induction. The base case ($n=1$) is covered by 
Lemma \ref{lem:bcforballsinsideballs}. 

For the inductive step, assume that our lemma holds for $n=i.$ In 
other words, suppose that 
 $$T^{-i}\left(\Bl{p^{-jk}}{y}\right) = \bigcup_{\nu =1}^{p^{ki}} B_\nu$$ 
where the $B_\nu$ are balls of radius $p^{-k(i+j)}$
one each in every ball of radius $p^{-ki}$ of $\ZZ_p.$ Then, 
$T^{-(i+1)}\left(\Bl{p^{-jk}}{y}\right) =\bigcup T^{-1} (B_\nu).$ 
But, by Lemma \ref{lem:bcforballsinsideballs}, for each $\nu,$ we 
have that $T^{-1}(B_\nu) = \bigcup_{\mu=1}^{p^k}C_\mu^\nu,$
where each $C_\mu^\nu$ is a ball of radius $p^{-k(i+j+1)},$ one in 
each ball of radius $p^k$ of $\ZZ_p.$ In particular, for a given 
$\nu,$ the $C_\mu^\nu$ are in distinct balls of radius $p^{-k(i+1)}.$

Finally, we argue that if $C_\mu^\nu$ and $C_{\mu'}^{\nu'}$ are 
contained in the same ball of radius $p^{-k(i+1)}$, 
then it follows that $\nu=\nu'$ (and so $\mu=\mu'$ by the above 
argument). 
Indeed, say $x\in C_\mu^\nu$ and $x'\in C_{\mu'}^{\nu'}$ with 
$|x-x'|\leq p^{-k(i+1)}.$ 
Then $|T(x)-T(x')|\leq p^{-ki}.$ But, $T(x)\in B_\nu$ and 
$T(x')\in B_{\nu'},$ so, since $B_\nu$ and $B_{\nu'}$ are in distinct 
balls of radius $p^{-ki}$ unless $\nu=\nu',$ we find that indeed 
$\nu=\nu'.$

We thus see that 
$$T^{-(i+1)}\left(\Bl{p^{-jk}}{y}\right)=\bigcup_{\nu=1}^{p^{kn}}
\bigcup_{\mu=1}^{p^k} C_\mu^\nu,$$ where the $C_\mu^\nu$ 
are balls of radius $p^{-k(i+j+1)},$ one each in every ball of radius
 $p^{-k(i+1)}$ of $\ZZ_p.$ 
The lemma thus holds for $n=i+1,$ completing the inductive step and 
therefore the proof.
\end{proof}

The preceding results suggest that if $T$ is $(p^{-k},p^k)$ locally 
scaling and $B$ is a ball, taking $T^{-n}(B)$ for larger and larger 
$n$ 
will produce sets that are very well mixed among $\ZZ_p.$ In fact, we
 have already established enough machinery to show that $T$ is  
\emph{mixing} with respect to the standard Haar measure $\mu$ on  
$\ZZ_p.$ The measure $\mu$ is completely specified by defining its  
values on balls: the $\mu$ measure of each ball is defined to be its  
radius. A transformation $T$ is said to be {\it mixing} if  for any  
$\mu$-measurable sets $U $ and $V$
 we have that $\lim_{n\to \infty}\mu(T^{-n}(U)\cap V) = \mu(U) 
 \mu(V).$
 It turns out, however, that it is enough to establish this in the 
 case that $U$ is a ball of radius $p^{-kj}$ 
 and $V$ is a ball of radius $p^{-k\ell}$ for any positive integers 
 $j$ and $\ell$ (see e.g. \cite[Theorem 6.3.4]{Silva08}). Taking 
 $n\geq \ell,$ we find that $T^{-n}(U)=\bigcup B_\nu$ 
 where each $B_\nu$ is a ball of radius $p^{-k(n+j)}$ one each in 
 every ball of radius $p^{-kn}$ in $\ZZ_p.$ But there are 
 $p^{k(n-\ell)}$ balls of radius $p^{-kn}$ in $B,$ each of them 
 containing one $B_\nu.$ 
 Therefore, \begin{align*}\mu(T^{-n}(U)\cap V)&= 
 \mu\left(\bigcup_{\nu=1}^{p^{k(n-\ell)}}B_{\nu}\right)=\sum 
 _{\nu=1}^{p^{k(n-\ell)}} \mu(B_{\nu})\\
 &=p^{k(n-\ell)}p^{-k(n+j)} =p^{-k(\ell+j),}   
     \end{align*} 
which is equal to $\mu(U)\mu(V).$ It follows that $T$ is mixing. 

We conclude this discussion with an example. 
\begin{ex}
Consider the maps $S,T:\ZZ_2\to\ZZ_2$ $S(x) = \binom{x}{2}$ and 
$T(x) =\binom{x}{3}.$
Let us consider the preimages of $2\ZZ_2$ and $1+2\ZZ_2$ (the two 
balls of radius $1/2$ in $\ZZ_2$) 
under $S$ and $T.$

Now, $S(x)=x(x-1)/2.$ One of $x$ and $x-1 $ will cancel out the 
factor of $2$ in the denominator.
 For $S(x)$ to be in $2\ZZ_2,$ we need another factor of $2$ in the  
numerator, meaning that either $x$ or $x-1$ must be equal to $0$  
modulo $4.$ In other words, 
 $S^{-1}(2 \ZZ_2) = 4 \ZZ_2 \cup (1+4\ZZ_2).$
 On the other hand, if
we want to have $S(x) \in 1+2\ZZ_2,$ 
neither $x$ nor $x-1$ can be congruent to $0$ modulo $4.$ In other 
words, $x\equiv 2\text{ or } 3\bmod 4.$ Thus, $S^{-1}(1+2\ZZ_2) = 
(2+4\ZZ_2)\cup (3+4\ZZ_2).$ We can thus see how taking the preimage 
mixes up among the two balls of radius $1/2.$

Let us now consider $T(x)=x(x-1)(x-2)/6.$ 
Notice that $T(x) \in 2\ZZ_2$ precisely if $x\equiv 0\bmod2.$ Indeed, 
we then get two powers of $2$ from $x$ and $x-2,$ one of which 
cancels a power of $2$ from the denominator. Therefore, $T^{-1}(2 
\ZZ_2)=2 \ZZ_2.$ Likewise, $T^{-1}(1+2\ZZ_2) = 1 + 2 \ZZ_2.$ So, in 
general, if $B$ is a ball of radius $1/2$ in $\ZZ_2,$ then for all 
$n>0,$ 
it is the case that $T^{-n}(B) = B.$

This discussion of preimages 
suggests that $S$ is mixing while $T$ is not, and, as we shall see, 
this is indeed the case.
\end{ex}

\subsection{Connections with Bernoulli Maps}
With our setup, we will be able to say that certain maps (namely the
$(p^{-k},p^k)$ locally scaling ones)  in many senses behave ``the 
same''
as the $p$-adic Bernoulli shift (or one of its iterates). To make
these notions 
more precise, we introduce some concepts in topological dynamics.

\begin{defn}
Two maps $T:\ZZ_p\lra\ZZ_p$ and $S:\ZZ_p\lra\ZZ_p$
are said to be (topologically) isomorphic if there exists  a 
homeomorphism $\Phi:\ZZ_p\lra \ZZ_p$ such that $\Phi\circ 
T(x)=S\circ\Phi(x)$ for
all $x.$    In other words, $\Phi$ conjugates the action of $T$ to the action of $S.$
\end{defn}

Close variants of the following theorem and proof   are in 
\cite{polybern}.

\begin{theorem}\label{thm:isombern}
Let $T:\ZZ_p\lra\ZZ_p$ be a $(p^{-k},p^k)$ locally scaling map. Then, 
$T$ is topologically  isomorphic to $\shift^k.$
\end{theorem} 
\begin{proof}
We must find a homeomorphism  $\Phi:\ZZ_p\lra\ZZ_p$ such that $\Phi 
\circ T=\shift^ k \circ \Phi.$ Consider $$\Phi(x) = \sum_{i=0}^\infty 
d_i(x) p^i, 
$$
where, if $i = q k + r$ for $0\leq r< k,$ we have that $d_i(x) = 
(T^q(x))_r$ (the $r$th digit in the $p$-adic expansion of $T^q(x)$). 
From this definition, it is easy to see that $\Phi \circ T(x)=\shift^ k 
\circ \Phi(x)$ for all $x\in \ZZ_p.$

We now proceed to show that $\Phi,$ as defined, is continuous and 
bijective.  Because $\ZZ_p$ is compact, continuity
of the inverse follows by general facts in topology. Therefore, 
we will have proved that $\Phi$  is a homeomorphism. 

To show that $\Phi$ is continuous, suppose that $|x - y| \leq p^{-k 
\eta}$ for $\eta \geq 1.$ It follows by local scaling that $|T(x) - 
T(y)| \leq p^{-k(\eta-1)},$ and in general,  we can see that for $q 
\leq \eta - 1,$ we have that $|T^q(x) - T^q(y)| \leq p^{-k}.$ 
 Therefore, the $p$-adic expansions of $\Phi(x)$ and $\Phi(y)$ 
agree at least up to the $p^{-(k\eta -1)}$
term, and we see that $|\Phi(x) - \Phi(y) | \leq p^{-k\eta}.$ Continuity 
follows.

To show injectivity, suppose that $\Phi(x)=\Phi(y).$ Then, $|T^i(x) - 
T^i(y)|\leq p^{-k}$ for all $i\geq0.$ But if $|T^i(x) - T^i(y)|\neq0$ 
for some $i,$ it 
follows by local scaling that there is a $j\geq0$ such that 
$|T^{i+j}(x) - T^{i+j}(y)|>p^{-k},$ a contradiction.

Finally, we prove surjectivity. Suppose that $y\in\ZZ_p,$ and we want 
to find an $x\in \ZZ_p$ such that $\Phi(x) = y.$ From the definition 
of $\Phi,$ we 
see that we are looking for an $x$ such that 
\begin{eqnarray*}
x &\in& \Bl{p^{-k}}{y}\\
T(x)&\in&\Bl{p^{-k}}{\shift^k(y)}\\
T^2(x)&\in&\Bl{p^{-k}}{\shift^{2k}(y)}\\
 &\vdots&\\
T^i(x)&\in&\Bl{p^{-k}}{\shift^{ik}(y)}\\
 &\vdots&
 \end{eqnarray*}
 Therefore, to prove that $\Phi$ is surjective, we have to show that 
for all $y\in \ZZ_p,$ the intersection  
$$\bigcap_{i\geq0}T^{-i}\left(\Bl{p^{-k}}{\shift^{ik}(y)}\right)$$ 
is nonempty.
 
 And indeed, by Lemma \ref{lem:ballsinsideballs}, for $i\geq1,$ each  
 ball of  $T^{-i}\left(\Bl{p^{-k}}{\shift^{ik}(y)}\right)$ 
 contains exactly one ball of  
 $T^{-(i+1)}\left(\Bl{p^{-k}}{\shift^{(i+1)k}(y)}\right).$ Furthermore, 
 one ball of $T^{-1}\left(\Bl{p^{-k}}{y}\right)$ is contained in 
 $\Bl{p^{-k}}{y}.$ These considerations show that an intersection of 
 any finite subcollection of the 
 $T^{-i}\left(\Bl{p^{-k}}{\shift^{ik}(y)}\right)$ is nonempty, and 
 therefore, by compactness of $\ZZ_p,$ we have that
  $$\bigcap_{i\geq0}T^{-i}\left(\Bl{p^{-k}}{\shift^{ik}(y)}\right)
\neq\emptyset,$$ as required.
\end{proof}

\subsection{``Small'' Perturbations on $\shift^k$}
At this point in the presentation, we have determined that 
$(p^{-k},p^k)$ maps are Bernoulli.  Furthermore, by understanding the 
Mahler expansion of the shift map $\shift^k,$ as we did earlier, we hope 
to 
understand the scaling properties of polynomial maps (namely, finite
$\ZZ_p$-linear combinations of the $\binom{x}{n}$), and thus, to find 
Bernoulli polynomials. As we will see, we have an infinite class of 
polynomial maps such that any $g$ in this class can be written as a 
sum of $\shift^k$ and a perturbing factor satisfying the Lipschitz 
property (see below).  In turn, this perturbing factor has small 
enough Lipschitz constant so that $g$ is $(p^{-k},p^k)$ locally 
scaling, like $\shift^k.$  We proceed to lay the groundwork for this 
argument.  
\begin{defn}
We say that a function $T:\ZZ_p\lra\ZZ_p$ is $C$-Lipschitz if 
$|T(x)-T(y)|\leq C|x-y|$ for all $x,y\in\ZZ_p.$ 
\end{defn}
If $T$ is $C$-Lipschitz and $a\in\QQ_p,$ it is clear that $aT$ is 
$|a|C$-Lipschitz. Also, because of the strong triangle inequality, if 
$T_i$ is $C_i$-Lipshitz, $\sum T_i$ is $(\sup_i\{C_i\})$-Lipshitz, 
provided this supremum  exists (i.e., the $C_i$ are bounded). 

Lipshitz maps are important in our discussion because they provide us 
with a way of slightly modifying a locally scaling function such that 
the resulting map is still locally scaling. The following proposition 
demonstrates one such method.  
\begin{prop}\label{prop:perturbscaling}
Let $T:\ZZ_p\lra\ZZ_p$ be $(r,C)$ locally scaling, $S:\ZZ_p\lra\ZZ_p$ 
be $D$-Lipshitz with $D<C,$ and suppose that $u\in\ZZ_p^\times.$ Then 
$\tilde T=uT+S$ is $(r,C)$-locally scaling.
\end{prop}
\begin{proof}
Take $x$ and $y$ with $|x-y|\leq r.$ Then 
\begin{eqnarray}
	|\tilde T(x)-\tilde T(y)|&=&|(uT(x)+S(x))-(uT(y)+S(y))|\\
&=&|u(T(x)-T(y))+(S(x)-S(y))|.
\end{eqnarray}
But, $|u(T(x)-T(y))|=C|x-y|>D|x-y|\geq |S(x)-S(y)|.$  Therefore, 
$|\tilde T(x)-\tilde T(y)|=C|x-y|,$ as desired.
\end{proof}

In particular, if $\tilde f(x)=\shift^k(x)+b(x)$ where $b$ is 
$p^{k-1}$-Lipshitz, then $\tilde f$ is $(p^{-k},p^k)$ locally 
scaling, and hence Bernoulli. This seemingly uninspiring condition is 
actually extremely powerful because it gives us  sufficient 
conditions on the Mahler expansion of a continuous map 
$T:\ZZ_p\lra\ZZ_p$ for the map to be isomorphic to $\shift^k$ for some 
$k.$  

Before stating the result, however, we need the following lemma, 
which we do not  prove here:
\begin{lemma}\label{lem:bclip}
The $\ZZ_p$ map $x\mapsto {x\choose n}$ is
$p^{\floor{\log_pn}}$-Lipshitz for all $n.$ 
\end{lemma}
\begin{proof}
See \cite[pg. 227]{robert}.
\end{proof}

\begin{defn}\label{cbc}
Let $T(x)$ be given by the uniformly converging Mahler series 
$$T(x)=\sum_{n=0}^\infty
A_n{x\choose n}$$ and set $C=\max_n p^{\floor{\log_pn}}|A_n|.$ 
Suppose   that the following conditions hold:

\begin{itemize}
\item There exists a unique $n_0$ which attains this maximum;
\item $n_0=p^k$ for some $k>0$; 
\item $|A_{n_0}|=1$ (thus, $C=n_0=p^k$).

\end{itemize}
Then, we say that $T$ is in the \emph{{\cbc}}. \end{defn}
\begin{theorem}\label{thm:cbls}
Let  $T$ be in the {\cbc}, with $C=p^k$ as in the definition. Then  
$T$ is $(p^{-k},p^k)$ locally scaling and hence isomorphic to $\shift^k.$
\end{theorem}

\begin{proof}
A variant of this theorem is in 
\cite{polybern}.  The proof that follows uses the Mahler expansion of $\shift^k,$
and is original to this paper.

The idea of the proof is that by the {\cbc} conditions, the 
$A_{p_k} \binom{x}{p^k}$ term will be the dominant one in determining the scaling
behavior of $T.$  Indeed, as we will see below, after multiplying by a unit as appropriate, 
$T$ and $\shift^k$ differ by a $p^{k-1}$-Lipshitz component; this follows by the {\cbc} conditions,
as well as the tight control guaranteed by Lemma \ref{lem:bclip}. The details of this argument follow. 

Recall that for the map $\shift^k,$ we have that  $a^{(k)}_{p^k}=1.$
 Therefore, the fact that $|A_{p_k}|=1$ in the 
Mahler expansion of $T$   implies that there exists
$u\in\ZZ_p^\times$ such that $u a_{p^k}^{(k)}=A_{p_k}.$ Consider
the general term $b_n=ua_n^{(k)}-A_n.$ By choice of $u,$
we have $b_{p^k}=0.$ Furthermore, the  strong triangle inequality tells us that
$|b_n|\leq\max\{|a_n^{(k)}|,A_n\}.$ Regardless of what the
maximum is for a particular $n\neq p^k,$
it is the case that $|b_n|p^{\floor{\log_pn}}<p^k$ by definition and Corollary
\ref{cor:shiftcoeffs}. Therefore, using Lemma \ref{lem:bclip}, we see
that $b_n{x\choose n}$ is $p^{k-1}$-Lipshitz for all $n$ (since
$b_{p^k}=0,$ the claim is trivial for $b=p^k$).

Hence, $u\shift^k(x)-T(x)=\sum_{n=0}^\infty b_n{x\choose n}$ is
$p^{k-1}$-Lipshitz. This implies, by Proposition 
\ref{prop:perturbscaling}, that $T$ is $(p^{-k},p^k) $ locally
scaling and the theorem follows.   
\end{proof}

\begin{rem}
The reader may (rightly) point out that proving the local scaling
properties of  maps like $\binom{x}{p^k}$ 
should not be much harder than proving Lipschitz properties of the 
$\binom{x}{n}.$ Indeed,   this was done directly in \cite{polybern}. 
Going through the Mahler expansion of the shift map
provides us with not as much a new proof of the local scaling 
properties, but an interpretation of the maps in the {\cbc}, which 
satisfy them.

\end{rem}

\section{Other Related Maps}\label{sec:related}
The $p$-adic shift is actually a special case of another natural 
class of maps. First off, define $g:\QQ_p\lra\ZZ_p$ as follows. Given
$x\in\QQ_p,$ we can express it uniquely as  $x=\sum_{i=\ell}^\infty
b_ip^i$ where $\ell\leq0$ and $b_i\in\{0,1,\ldots,p-1\}$ for all $i.$

 With our setup notation, we let $g(x)=\sum_{i=0}^\infty b_ip^i.$ In 
 other words, $g$ chops off the negative powers of $p$ in the 
 $p$-adic expansion of $x.$

Take $a\in\QQ_p.$ Define $f_a:\ZZ_p\lra\ZZ_p$ by $f_a(x)=g(ax)$ where 
$g$ is as above. Notice that the $p$-adic shift is $f_a$ where 
$a=1/p.$ 

We will mostly be interested in $f_a$ with $|a|>1.$ Now, $|a|>1$ 
implies that $a=a'/p^k$ for $a'\in \ZZ_p^\times$ and $k>0.$
Therefore, for $x\in\ZZ_p,$ $f_a(x)=g(a'x/p^k)$ where $a'x\in\ZZ_p$ 
and so, $f_a(x)=\shift^k(a'x).$ 

We now make a few remarks about the topological dynamics of the maps 
$f_a.$


For $a$ with $|a|<1,$ we have that $f_a(\ZZ_p)\subsetneq\ZZ_p,$ so
$f_a$ is not surjective, and 
cannot be isomorphic to the shift.

 The situation when $|a|>1$ is quite different.

\begin{theorem}\label{thm:famp}
For $|a|=p^k,$ with $k>0,$ $f_a$ is $(p^{-k},p^k)$ locally scaling, 
and hence isomorphic to  $\shift^k.$
\end{theorem}
\begin{proof}
We know that $f_a(x)=\shift^k(a'x)$ where $a'\in\ZZ_p^\times.$  Take $x$ 
and $y$ with $|x-y|\leq p^{-k}.$ Then, $|a'x-a'y|=|x-y|\leq p^{-k}.$ 
Therefore, $|\shift^k(a'x)-\shift^k(a'y)|=p^k|a'x-a'y|=p^k|x-y|$ and so $f_a$ 
is $(p^{-k},p^k)$ locally scaling. 
\end{proof}

\bibliographystyle{amsplain}
\bibliography{ergodictheory}
\end{document}